\documentclass[11pt]{amsart}

\makeatletter

\def\Dm{\mathbf{D}_m}
\def\IDm{\mathbf{ID}_m}
\def\lach{\Psi}
\makeatother

\usepackage{hyperref}
\usepackage{amsthm,amssymb}

\theoremstyle{theorem}
\newtheorem{theorem}{Theorem}[section]
\newtheorem{proposition}{Proposition}[section]
\newtheorem{corollary}{Corollary}[section]
\newtheorem{lemma}{Lemma}[section]

\theoremstyle{definition}
\newtheorem{problem}{Problem}

\begin{document}

\title[Isomorphism types of Rogers semilattices]{Isomorphism types of Rogers semilattices in the analytical hierarchy}


\author{Nikolay Bazhenov}

\address{Sobolev Institute of Mathematics\\
4 Acad. Koptyug Ave., Novosibirsk, 630090, Russia\\
and Department of Mathematics, School of Sciences and Humanities\\
Nazarbayev University\\
53 Qabanbaybatyr Ave.,
Nur-Sultan, 010000, Kazakhstan}
\email{bazhenov@math.nsc.ru}

\author{Sergey Ospichev}

\address{Sobolev Institute of Mathematics\\
4 Acad. Koptyug Ave., Novosibirsk, 630090, Russia\\
and Novosibirsk State University\\
2 Pirogova St., Novosibirsk, 630090, Russia}
\email{ospichev@math.nsc.ru}

\author{Mars Yamaleev}

\address{Kazan Federal University\\
18 Kremlevskaya St., Kazan, 420008, Russia}
\email{mars.yamaleev@kpfu.ru}

\thanks{The work of N.~Bazhenov was supported by Nazarbayev University Faculty Development Competitive Research Grants N090118FD5342. The work of S.~Ospichev was funded by RFBR according to the research project No.~17-01-00247. The work of M.~Yamaleev was supported by Russian Science Foundation, project No.~18-11-00028.}

\begin{abstract}
	A numbering of a countable family $S$ is a surjective map from the set of natural numbers $\omega$ onto $S$. A numbering $\nu$ is reducible to a numbering $\mu$ if there is an effective procedure which given a $\nu$-index of an object from $S$, computes a $\mu$-index for the same object. The reducibility between numberings gives rise to a class of upper semilattices, which are usually called Rogers semilattices. The paper studies Rogers semilattices for families $S \subset P(\omega)$ belonging to various levels of the analytical hierarchy. We prove that for any non-zero natural numbers $m\neq n$, any non-trivial Rogers semilattice of a $\Pi^1_m$-computable family cannot be isomorphic to a Rogers semilattice of a $\Pi^1_n$-computable family. One of the key ingredients of the proof is an application of the result by Downey and Knight on degree spectra of linear orders.
\end{abstract}


\maketitle

\section{Introduction}

Uniform computations for families of mathematical objects constitute a classical line of research in recursion theory. An important approach to studying such computations is provided by the theory of numberings. The theory goes back to the ground-breaking work of G{\"o}del \cite{Goedel}, where an effective numbering of first-order formulae was used in the proof of the incompleteness theorems. One of the first results in the theory of computable numberings was obtained by Kleene~\cite{Kleene}: he gave a construction of a universal partial computable function. After that, the foundations of the modern theory of numberings were developed by Kolmogorov and Uspenskii~\cite{KU,Usp-55} and, independently, by Rogers~\cite{Rogers}.

Let $\mathcal{S}$ be a countable set. A \emph{numbering} of $\mathcal{S}$ is a surjective map $\nu$ from the set of natural numbers $\omega$ onto $\mathcal{S}$. A standard tool for comparing the complexity of different numberings is provided by the notion of \emph{reducibility} between numberings: A numbering $\nu$ is \emph{reducible} to another numbering $\mu$, denoted by $\nu\leq \mu$, if there is total computable function $f(x)$ such that $\nu(x) = \mu(f(x))$ for all $x\in\omega$. In other words, there is an effective procedure which, given a $\nu$-index of an object from $\mathcal{S}$, computes a $\mu$-index for the same object.

In this paper, we consider only families $\mathcal{S}$ containing subsets of $\omega$, i.e., we always assume that $\mathcal{S} \subseteq P(\omega)$ and $\mathcal{S}$ is countable.

Suppose that $\mathcal{S}$ is a family of computably enumerable (c.e.) sets. A numbering $\nu$ of the family $\mathcal{S}$ is \emph{computable} if the set
\begin{equation} \label{equ:g-nu}
	G_{\nu} = \{ \langle n,x\rangle \,\colon x\in \nu(n)\}
\end{equation}
is c.e. A family $\mathcal{S}$ is \emph{computable} if it has a computable numbering.

In a standard recursion-theoretic way, the notion of reducibility between numberings gives rise to the \emph{Rogers upper semilattice} (or \emph{Rogers semilattice} for short) of a family: For a computable family $\mathcal{S}$, this semilattice contains the degrees of all computable numberings of $\mathcal{S}$. As per usual, here two numberings have the same degree if they are reducible to each other.

There is a large body of literature on Rogers semilattices of computable families. To name only a few, computable numberings were studied by Badaev~\cite{Bad-77,Bad-94}, Ershov~\cite{Ershov-68,Ershov-Book}, Friedberg~\cite{Friedberg}, Goncharov~\cite{Gon-80,Gon-83}, Khutoretskii~\cite{Khut-71}, Lachlan~\cite{Lach-64,Lach-65}, Mal'tsev~\cite{Mal-65}, Pour-El~\cite{Pour-El}, Selivanov~\cite{Sel-76}, and many other researchers.

Goncharov and Sorbi~\cite{GS-97} started developing the theory of generalized computable numberings. One of their approaches to generalized computations can be summarized as follows. Let $\Gamma$ be a complexity class (e.g., $\Sigma^0_1$, $d$-$\Sigma^0_1$, $\Sigma^0_n$, or $\Pi^1_n$). A numbering $\nu$ of a family $\mathcal{S}$ is \emph{$\Gamma$-computable} if the set $G_{\nu}$ from \eqref{equ:g-nu} belongs to the class $\Gamma$. We say that a family $\mathcal{S}$ is \emph{$\Gamma$-computable} if it has a $\Gamma$-computable numbering. Note that the classical notion of a computable numbering becomes a synonym of a $\Sigma^0_1$-computable numbering.

In a similar way to computable numberings, one can introduce the notion of the Rogers semilattice of a $\Gamma$-computable family, see Section~\ref{subsect:numb} for the details. One of natural questions in the area of generalized computable numberings can be formulated as follows.

\begin{problem} \label{problem:main}
	Given a complexity class $\Gamma$, study the isomorphism types of Rogers semilattices of $\Gamma$-computable families.
\end{problem}

We give a short overview of some results related to Problem~\ref{problem:main}, further related work is discussed in Section~\ref{subsect:related}. We say that an upper semilattice is a \emph{Rogers $\Gamma$-semilattice} if it is isomorphic to the Rogers semilattice of a $\Gamma$-computable family of sets. A semilattice is \emph{non-trivial} if it contains more than one element.

Let $n$ be a non-zero natural number. Badaev, Goncharov, and Sorbi~\cite{BGS-06} proved that for any $m \geq n + 3$, any non-trivial Rogers $\Sigma^0_m$-semilattice is not isomorphic to a Rogers $\Sigma^0_n$-semilattice. Podzorov~\cite{Podz} obtained a generalization of this theorem: he showed that a similar result holds for $m \geq n + 2$. It is still open whether this fact is true for $m=n+1$.

Badaev and Goncharov~\cite{BG-08} extended the result of~\cite{BGS-06} to the hyperarithmetical hierarchy: They showed that for any computable ordinals $\alpha > 0$ and $\beta\geq \alpha+3$, any non-trivial Rogers $\Sigma^0_{\beta}$-semilattice is not isomorphic to a Rogers $\Sigma^0_{\alpha}$-semilattice.

We note that the situation in the Ershov hierarchy is quite the opposite. Recall that for a non-zero natural number $n$, $\Sigma^{-1}_n$ denotes the class of all $n$-c.e. sets. The result of Herbert, Jain, Lempp, Mustafa, and Stephan (Theorem~2 in~\cite{HJLMS}) implies that every Rogers $\Sigma^{-1}_n$-semilattice is also a Rogers $\Sigma^{-1}_{n+1}$-semilattice.

In this paper, we continue the investigations of \cite{BGS-06,Podz,BG-08} in the setting of the analytical hierarchy. We prove that for any natural numbers $m>n\geq 1$, a non-trivial Rogers $\Pi^1_m$-semilattice is not isomorphic to a Rogers $\Pi^1_n$-semilattice (Theorem~\ref{theo:main}). Quite unexpectedly, one of the main ingredients of the proof is an application of the result of Downey and Knight~\cite{DK-92} on degree spectra of linear orders.

The structure of the paper is as follows. Section~\ref{sect:prelim} contains the necessary preliminaries and a discussion of related work. In Section~\ref{sect:main-res}, we prove Theorem~\ref{theo:main}: First, we give a complete outline of the proof including all key ideas. After that, the proofs of the auxiliary claims are given in separate subsections.


\section{Preliminaries} \label{sect:prelim}

We use lowercase bold Latin letters (e.g., $\mathbf{a}$, $\mathbf{b}$, $\mathbf{c}$) to denote $m$-deg\-rees. Lowercase bold Latin letters with a subscript $T$ (e.g., $\mathbf{x}_T$, $\mathbf{y}_T$, $\mathbf{z}_T$) denote Turing degrees.

We assume that for any considered countable structure, its universe is contained in $\omega$. For a structure $\mathcal{A}$, $D(\mathcal{A})$ denotes the atomic diagram of $\mathcal{A}$.

We treat upper semilattices as structures in the language $L_{usl} = \{ \leq, \vee\}$. If $\mathcal{A}$ is an upper semilattice and $a\leq_{\mathcal{A}} b$ are elements from $\mathcal{A}$, then the \emph{interval} $[a;b]_{\mathcal{A}}$ is the semilattice
\[
	[a;b]_{\mathcal{A}} := ( \{ c\in \mathcal{A} \,\colon a \leq_{\mathcal{A}} c \leq_{\mathcal{A}} b\}; \leq_{\mathcal{A}}, \vee_{\mathcal{A}} ).
\]

For a complexity class $\Gamma$, we say that an $L_{usl}$-structure $\mathcal{A}=(\omega;\leq_A,\vee_A)$ is a \emph{$\Gamma$-presentation} of an upper semilattice $\mathcal{M}$ if $\mathcal{A}$ satisfies the following conditions:
\begin{enumerate}
	\item the function $\vee_A$ is total computable,
	
	\item the relation $\sim_A \,:=\, \leq_A \, \cap\, \geq_A$ is a congruence on $\mathcal{A}$,
	
	\item the relations $\leq_{A}$ and $\sim_A$ both belong to the class $\Gamma$,
	
	\item the quotient structure $\mathcal{A}/_{\displaystyle{\sim_A}}$ is isomorphic to $\mathcal{M}$.
\end{enumerate}


\subsection{Numberings} \label{subsect:numb}

Suppose that $\nu$ is a numbering of a family $\mathcal{S}_0$, and $\mu$ is a numbering of a family $\mathcal{S}_1$. Note that the condition $\nu \leq \mu$ always implies that $\mathcal{S}_0 \subseteq \mathcal{S}_1$.

Numberings $\nu$ and $\mu$ are \emph{equivalent} (denoted by $\nu\equiv \mu$) if $\nu\leq \mu$ and $\mu \leq \nu$. The numbering $\nu\oplus\mu$ of the family $\mathcal{S}_0\cup\mathcal{S}_1$ is defined as follows:
\[
	(\nu \oplus \mu)(2x) = \nu(x), \quad (\nu \oplus \mu)(2x+1) = \mu(x).
\]
The following fact is well-known (see, e.g., p.~36 in \cite{Ershov-Book}): If $\xi$ is a numbering of a family $\mathcal{T}$, then
\[
	(\nu \leq  \xi \,\&\, \mu \leq  \xi) \ \Leftrightarrow\ (\nu\oplus \mu \leq  \xi).
\]
For further background on numberings, the reader is referred to, e.g., \cite{Ershov-Book,Ershov-99,BG-00}.

Let $\Gamma$ be a complexity class with the following properties:
\begin{itemize}
	\item[(a)] If $\nu$ is a $\Gamma$-computable numbering and $\mu$ is a numbering such that $\mu \leq \nu$, then $\mu$ is also $\Gamma$-computable.
	
	\item[(b)] If numberings $\nu$ and $\mu$ are both $\Gamma$-computable, then the numbering $\nu\oplus \mu$ is also $\Gamma$-computable.
\end{itemize}
For example, it is not hard to show that for any non-zero natural number $n$, each of the classes $\Sigma^0_n$, $\Sigma^{-1}_n$, and $\Pi^1_n$ has these properties.

Consider a $\Gamma$-computable family $\mathcal{S}$. By $Com_{\Gamma}(\mathcal{S})$ we denote the set of all $\Gamma$-computable numberings of $\mathcal{S}$. Since the relation $\equiv$ is a congruence on the structure $(Com_{\Gamma}(\mathcal{S}); \leq, \oplus)$, we use the same symbols $\leq$ and $\oplus$ on numberings of $\mathcal{S}$ and on $\equiv$-equivalence classes of these numberings.

The quotient structure $\mathcal{R}_{\Gamma}(\mathcal{S}) := (Com_{\Gamma}(\mathcal{S}) /_{\displaystyle{\equiv}}; \leq, \oplus)$ is an upper semilattice. We say that $\mathcal{R}_{\Gamma}(\mathcal{S})$ is the \emph{Rogers semilattice} of the $\Gamma$-computable family $\mathcal{S}$. For the sake of convenience, we use the following notation:
\[
	\mathcal{R}^1_n(\mathcal{S}) := \mathcal{R}_{\Pi^1_n}(\mathcal{S}).
\]

\subsection{Related work} \label{subsect:related}

The questions related to Problem~\ref{problem:main} were extensively studied for Rogers $\Sigma^0_n$-semilattices. In particular, the following results on the number of isomorphism types of Rogers $\Sigma^0_n$-semilattices are known. Ershov and Lavrov~\cite{EL-73} (see also p.~72 in~\cite{Ershov-Book}, and \cite{Ershov-03}) showed that there are finite families $\mathcal{S}_i$, $i\in\omega$, of c.e. sets such that the semilattices $\mathcal{R}_{\Sigma^0_1}(\mathcal{S}_i)$ are pairwise non-isomorphic. In other words, there are infinitely many isomorphism types of Rogers $\Sigma^0_1$-semilattices. V'yugin \cite{V'yugin} proved that there are infinitely many pairwise elementarily non-equivalent Rogers $\Sigma^0_1$-semilattices. Badaev, Goncharov, and Sorbi~\cite{BGS-05} proved that for any natural number $n\geq 2$, there are infinitely many pairwise elementarily non-equivalent Rogers $\Sigma^0_n$-semilattices. The reader is referred to, e.g., \cite{GS-97,BGPS-03,BGS-03,BMY} for further results on Rogers $\Sigma^0_n$-semilattices.

Rogers semilattices in the analytical hierarchy were previously studied by Dorzhieva~\cite{Dor-14,Dor-16}. In particular, she showed that for any non-zero $n$ and any non-trivial Rogers $\Pi^1_n$-semilattice $\mathcal{R}$, the first-order theory of $\mathcal{R}$ is hereditarily undecidable (Theorem~3 in \cite{Dor-16}).

Kalimullin, Puzarenko, and Faizrakhmanov~\cite{KPF-18} considered \emph{computable $\Pi^1_1$-numberings}. A \emph{$\Pi^1_1$-numbering} of a family $\mathcal{S}$ is a \emph{partial} map $\nu$ acting from $\omega$ onto $\mathcal{S}$ such that the domain of $\nu$ is enumeration reducible to the $\Pi^1_1$-complete set $\mathcal{O}$. A $\Pi^1_1$-numbering $\nu$ is \emph{computable} if the set
\[
	G^{\ast}_{\nu} = \{ \langle n,x\rangle \,\colon n\in dom(\nu),\ x\in \nu(n)\}
\]
is enumeration reducible to $\mathcal{O}$.


\section{Main Result} \label{sect:main-res}

\begin{theorem} \label{theo:main}
	Let $m > n$ be non-zero natural numbers. If $\mathcal{R}$ is a non-tri\-vi\-al Rogers $\Pi^1_m$-semilattice, then $\mathcal{R}$ is not isomorphic to a Rogers $\Pi^1_n$-se\-mi\-lat\-tice.
\end{theorem}
\begin{proof}
	Here we give a complete outline of the proof. This includes all of its main ingredients, but the proofs of some auxiliary statements are given in separate subsections.
	
	First, we give an upper bound on the complexity of (atomic diagrams of) arbitrary intervals in a Rogers $\Pi^1_n$-semilattice. Let $\mathcal{E}_n$ be the $\Sigma^1_n$-complete set.
	
	\begin{proposition}\label{prop:upper-bound}
		Let $\mathcal{M}$ be a Rogers $\Pi^1_n$-semilattice. If $a\leq_{\mathcal{M}} b$ are elements from $\mathcal{M}$, then the interval $[a;b]_{\mathcal{M}}$ has a $\Sigma^0_2(\mathcal{E}_n)$-presentation.
	\end{proposition}

Second, we obtain a family of upper semilattices which are realizable as intervals in every non-trivial Rogers $\Pi^1_k$-semilattice.

Let $\Dm$ denote the semilattice of $m$-degrees. As per usual, we assume that $\Dm$ does not contain the degrees $\deg_m(\emptyset)$ and $\deg_m(\omega)$. For a non-empty set $A\subset \omega$, let $\Dm(\leq A)$ be the principal ideal $\{ \mathbf{d}\,\colon \mathbf{d} \leq \deg_m(A)\}$ in the semilattice $\Dm$.

Suppose that $\nu$ and $\mu$ are $\Pi^1_k$-computable numberings of a family $\mathcal{S}$. We say that $\nu$ is \emph{$\Delta^1_k$-reducible} to $\mu$, denoted by $\nu \leq_{\Delta^1_k} \mu$, if there is a total $\Delta^1_k$ function $f(x)$ such that $\nu(x) = \mu(f(x))$, for all $x\in\omega$. The numberings $\nu$ and $\mu$ are \emph{$\Delta^1_k$-equivalent}, denoted by $\nu \equiv_{\Delta^1_k} \mu$, if $\nu\leq_{\Delta^1_k} \mu$ and $\mu\leq_{\Delta^1_k} \nu$.

\begin{proposition}\label{prop:init-segments}
	Let $k$ be a non-zero natural number, and $\mathcal{S}$ be a $\Pi^1_k$-com\-pu\-table family such that $\mathcal{S}$ contains at least two elements. Suppose that a set $U\subseteq\omega$ is immune and $\Delta^1_k$. Then for any $\Pi^1_k$-computable numbering $\alpha$ of the family $\mathcal{S}$, there is a $\Pi^1_k$-computable numbering $\beta$ of $\mathcal{S}$ with the following properties:
	\begin{itemize}
		\item[(a)] $\beta \equiv_{\Delta^1_k} \alpha$.
		
		\item[(b)] If $\mathcal{S}$ is finite, then the principal ideal induced by $\beta$ inside $\mathcal{R}^1_k(\mathcal{S})$ (we denote this ideal by $[\widehat \beta]^1_k$) is isomorphic to the structure $\Dm(\leq U)$.
		
		\item[(c)] If $\mathcal{S}$ is infinite, then the ideal $[\widehat\beta]^1_k$ is isomorphic to the structure $\Dm(\leq U)\setminus\{\mathbf{0}\}$, i.e. to the ideal $\Dm(\leq U)$ with the least element omitted.
	\end{itemize}
\end{proposition}

Proposition~\ref{prop:init-segments} allows us to find plenty of linearly ordered intervals inside an arbitrary non-trivial Rogers $\Pi^1_k$-semilattice. This is obtained via the following result:

\begin{proposition}\label{prop:immune}
	Suppose that $\mathcal{L}$ is a countable linear order with least and greatest elements. Then there exists a set $A \subseteq \omega$ with the following properties:
	\begin{enumerate}
		\item the principal ideal $\Dm(\leq A)$ is isomorphic to $\mathcal{L}$,
		
		\item $A$ is immune, and
		
		\item $A$ is c.e. in $(D(\mathcal{L}) \oplus \emptyset^{(2)})$.
	\end{enumerate}
\end{proposition}

Now we are ready to introduce the last key ingredient, and to finish the proof. The last ingredient uses the result of Downey and Knight~\cite{DK-92} on the degree spectra of linear orders.

Suppose that $\mathcal{M}$ is a countable infinite structure, and $\alpha$ is a computable ordinal. A Turing degree $\mathbf{x}_T$ is the \emph{$\alpha$th jump degree} of $\mathcal{M}$ if it is the least degree in the set
\[
	\{ \deg_T(D(\mathcal{A}))^{(\alpha)}\,\colon \mathcal{A}\cong \mathcal{M}, \text{ and the domain of } \mathcal{A} \text{ is } \omega  \}.
\]
A degree $\mathbf{x}_T$ is \emph{proper $\alpha$th jump degree} of $\mathcal{M}$ if $\mathbf{x}_T$ is the $\alpha$th jump degree of $\mathcal{M}$ and for any computable ordinal $\beta < \alpha$, $\mathcal{M}$ has no $\beta$th jump degree .

\begin{theorem} \label{theo:DK}
	\textnormal{(Downey and Knight~\cite{DK-92}, see also \cite{Kni86,AJK90})}
	For any computable ordinal $\alpha \geq 2$ and any Turing degree $\mathbf{x}_T\geq \mathbf{0}^{(\alpha)}$, there is a linear order having proper $\alpha$th jump degree $\mathbf{x}_T$.
\end{theorem}

\begin{corollary}\label{corol:01}
	For any Turing degree $\mathbf{x}_T$, there exists a linear order $\mathcal{L}$ with the following properties: $\deg_T(D(\mathcal{L})) \leq \textbf{x}_T^{(3)}$, and there is no copy $\mathcal{A}\cong\mathcal{L}$ with $\deg_T(D(\mathcal{A}))\leq \mathbf{x}_T$.
\end{corollary}

In Corollary~\ref{corol:01}, the desired structure $\mathcal{L}$ is the linear order from Theorem~\ref{theo:DK} having proper second jump degree $\mathbf{x}_T^{(3)}$.

Recall that $m>n\geq 1$. Towards contradiction, assume that $\mathcal{R}$ is a non-trivial Rogers $\Pi^1_m$-semilattice which is also a Rogers $\Pi^1_n$-semilattice.

Let $\mathbf{e}_{T}$ be the Turing degree of the set $\mathcal{E}_n$. By applying Corollary~\ref{corol:01} to the degree $\mathbf{x}_T = \mathbf{e}^{(2)}_T$, one can choose a linear order $\mathcal{L}$ such that $\mathcal{L}$ has least and greatest elements, $\deg_T(D(\mathcal{L})) \leq \mathbf{e}^{(5)}_T$, and there is no copy $\mathcal{A}$ of $\mathcal{L}$ with $\deg_T(D(\mathcal{A}))\leq \mathbf{e}^{(2)}_T$.

By Proposition~\ref{prop:immune}, one can find an immune set $A\subset\omega$ such that $\Dm(\leq A) \cong \mathcal{L}$ and $A$ is c.e. in $(D(\mathcal{L})\oplus \mathbf{0}^{(2)})$. In particular, this implies that the set $A$ is $\Sigma^0_1(\mathbf{e}_T^{(5)})$. Since the set $\mathcal{E}_n$ is $\Delta^1_m$, the set $A$ is also $\Delta^1_m$.

Since the Rogers $\Pi^1_m$-semilattice $\mathcal{R}$ is non-trivial, the corresponding $\Pi^1_m$-com\-putable family $\mathcal{S}$ contains more than one element. We apply Proposition~\ref{prop:init-segments} to $\mathcal{S}$ and the immune $\Delta^1_m$ set $A$, and find an interval $[a;b]_{\mathcal{R}}$ which is isomorphic either to $\mathcal{L}$, or to the structure $\mathcal{L}\setminus\{ 0\}$ (i.e. the order $\mathcal{L}$ with the least element omitted). W.l.o.g., one may assume that $[a;b]_{\mathcal{R}}$ is a copy of $\mathcal{L}$.

Recall that we assumed that $\mathcal{R}$ is a Rogers $\Pi^1_n$-semilattice. Thus, by Proposition~\ref{prop:upper-bound}, the interval $[a;b]_{\mathcal{R}}$ has a $\Sigma^0_2(\mathcal{E}_n)$-presentation $\mathcal{P}$. Using the presentation $\mathcal{P}$, it is straightforward to build a linear order $\mathcal{A}_0\cong\mathcal{L}$ such that $D(\mathcal{A}_0) \leq_T \mathbf{e}_T^{(2)}$. This contradicts the choice of the order $\mathcal{L}$. Therefore, $\mathcal{R}$ cannot be a Rogers $\Pi^1_n$-semilattice.

This concludes the (outline of the) proof of Theorem~\ref{theo:main}. The proofs of Propositions \ref{prop:upper-bound}--\ref{prop:immune} are given below.
\end{proof}


\subsection{Proof of Proposition~\ref{prop:upper-bound}}

We follow the lines of the proof of Lemma~7 in \cite{Podz}. For the sake of self-com\-pleteness, here we give a detailed exposition.

Let $\nu$ be a numbering of a family $\mathcal{S}$. The operator
\[
	\Psi(\nu; \cdot) \colon \{ W \subseteq \omega\,\colon W \text{ is c.e.},\ W\neq \emptyset\} \to \{ [\mu]_{\equiv}\,\colon \mu \leq \nu \}
\]
is defined as follows. Let $W$ be a non-empty c.e. set. Choose a total computable function $f(x)$ such that $range(f)=W$, and define $\Psi(\nu;W)$ as the $\equiv$-class of the numbering $\nu\circ f$. It is not hard to show that the operator $\Psi$ is well-defined, i.e. the value $\Psi(\nu;W)$ does not depend on the choice of a function $f$. Moreover, if $\mu$ is a numbering from the class $\Psi(\nu;W)$, then $range(\mu) = \nu(W) = \{ \nu(x)\,\colon x\in W\}$.

The operator $\Psi$ has the following properties (see p.~1120 in~\cite{Podz}):
\begin{itemize}
	\item[(a)] For any numbering $\mu$, the condition $\mu \leq \nu$ holds iff there is a c.e. set $W$ with $[\mu]_{\equiv} = \Psi(\nu;W)$.
	
	\item[(b)] For any non-empty c.e. sets $U$ and $V$, we have $\Psi(\nu;U\cup V) = \Psi(\nu;U) \vee \Psi(\nu;V)$.
	
	\item[(c)] $\Psi(\nu;U) \leq \Psi(\nu;V)$ if and only if $\nu(U) \subseteq \nu(V)$ and there is a partial computable function $\theta(x)$ such that $U\subseteq dom(\theta)$, $\theta(U)\subseteq V$, and for every $x\in U$, we have $\nu(x)=\nu(\theta(x))$.
\end{itemize}

Recall that $\mathcal{M}$ is a Rogers $\Pi^1_n$-semilattice. Hence, one can choose a $\Pi^1_n$-com\-pu\-table family $\mathcal{S}$, and identify an arbitrary element $c\in\mathcal{M}$ with a class $[\xi]_{\equiv}$, where $\xi$ is a $\Pi^1_n$-computable numbering of the family $\mathcal{S}$.

Choose a numbering $\nu$ of $\mathcal{S}$ such that $b=[\nu]_{\equiv}$. By the property~(a) of the operator $\Psi$, there is a non-empty c.e. set $U$ such that $a=\Psi(\nu;U)$. We define an $L_{usl}$-structure $\mathcal{A} = (\omega;\leq_A,\vee_A)$ as follows:
\begin{enumerate}
	\item $\vee_A$ is a total computable function such that $W_{x\vee_A y} = W_x \cup W_y$ for all $x$ and $y$;
	
	\item $x \leq_A y$ iff $\Psi(\nu; U\cup W_x) \leq \Psi(\nu; U\cup W_y)$.
\end{enumerate}

The properties of $\Psi$ imply that the relation $\sim_A\, := (\leq_A\, \cap \, \geq_A)$ is a congruence on $\mathcal{A}$, and the function $F\colon [e]_{\sim_A} \mapsto \Psi(\nu;U\cup W_e)$ is an isomorphism from $\mathcal{A}/_{\displaystyle \sim_{A}}$ onto $[a;b]_{\mathcal{M}}$.

In order to finish the proof, now it is sufficient to show that the relation $\leq_{A}$ is $\Sigma^0_{2}(\mathcal{E}_n)$. By the property~(c) of the operator $\Psi$, the condition $x\leq_A y$ is equivalent to
\begin{multline}
	(\exists i \in\omega) [ U\cup W_x \subseteq dom(\varphi_i) \ \&\  \varphi_i(U\cup W_x) \subseteq U\cup W_y \ \&\\
	(\forall z \in U\cup W_x)(\nu(z) = \nu(\varphi_i(z)))]. \label{equ:prop01}
\end{multline}
Since the condition $\nu(z) = \nu(v)$ can be rewritten as
\[
	(\forall k\in\omega)[ k\in \nu(z) \leftrightarrow k\in\nu(v) ]
\]
and $\nu$ is a $\Pi^1_n$-computable numbering, the equation~\eqref{equ:prop01} describes a $\Sigma^0_2(\mathcal{E}_n)$ condition. Therefore, the semilattice $\mathcal{M}$ has a $\Sigma^0_2(\mathcal{E}_n)$-presentation.
Proposition~\ref{prop:upper-bound} is proved.


\subsection{Proof sketch for Proposition~\ref{prop:init-segments}}

 The proof of this result is entirely based on Theorem 2 from the paper~\cite{Podz-03} by Podzorov.

Let $m$ be a non-zero natural number. A numbering $\nu$ is \emph{$\Delta^0_m$-reducible} to a numbering $\mu$, denoted by $\nu \leq_{\Delta^0_m} \mu$, if there is a total $\Delta^0_m$ function $f(x)$ such that $\nu(x) = \mu(f(x))$, for every $x$. We write $\nu \equiv_{\Delta^0_m} \mu$ if $\nu\leq_{\Delta^0_m} \mu$ and $\mu\leq_{\Delta^0_m} \nu$.

If $\nu$ is a $\Sigma^0_m$-computable numbering of a family $\mathcal{S}$, then by $[\widehat \nu]^0_m$ we denote the principal ideal induced by $\nu$ inside the Rogers $\Sigma^0_m$-semilattice of $\mathcal{S}$ (see \cite{Podz-03} for more details).

 \begin{theorem} \label{theo:Podzorov}
 \textnormal{(Podzorov~\cite{Podz-03})}
 Suppose that $m$ and $n$ are natural numbers such that $2\le m\le n$. Suppose also that a set $U\subseteq\omega$ is immune and $\Delta^0_m$. Let $\mathcal{S}$ be a $\Sigma^0_n$-computable family such that $\mathcal{S}$ contains at least two elements.   Then for any $\Sigma^0_n$-computable numbering $\alpha$ of the family $\mathcal{S}$, there is a $\Sigma^0_n$-computable numbering $\beta$ of $\mathcal{S}$ with the following properties:
 \begin{itemize}
 	\item[(a)] $\beta \equiv_{\Delta^0_m} \alpha$.
 	
 	\item[(b)] If $\mathcal{S}$ is finite, then $[\widehat \beta]^0_n$ is isomorphic to the structure $\Dm(\leq U)$.
 	
 	\item[(c)] If $\mathcal{S}$ is infinite, then the ideal $[\widehat\beta]^0_n$ is isomorphic to the structure $\Dm(\leq U)\setminus\{\mathbf{0}\}$.
 \end{itemize}
\end{theorem}
 A careful step-by-step analysis of the proof of Theorem~\ref{theo:Podzorov} shows that this result works well for analytical sets:  Roughly speaking, the construction described in~\cite{Podz-03} depends only on the complexity of the set $U$, and the desired numbering $\beta$ can be considered as a $\Delta^0_m$ permutation (or in our case, as a $\Delta^1_k$ permutation) of the numbering $\alpha$. Thus, one can recover the proof of Proposition~\ref{prop:init-segments} from the Podzorov's construction just by using some basic properties of the analytical hierarchy and the Tarski--Kuratowski algorithm.



\subsection{Proof of Proposition~\ref{prop:immune}}

Before giving the proof, we briefly discuss some known related results. The results of Lachlan~\cite{Lach70} imply that for any countable linear order $\mathcal{L}$ with least element, there exists an initial segment of $\Dm$ isomorphic to $\mathcal{L}$. Ershov~\cite{Ershov75a} gave a characterization of the semilattice $\Dm$ up to isomorphism, as a $c$-uni\-ver\-sal upper semilattice of cardinality continuum, see also \cite{Ershov-Book,Pal75}.

Let $\IDm$ denote the upper semilattice of immune $m$-degrees. Following~\cite{Mal85}, here we assume that the $m$-degree of a non-empty finite set also belongs to $\IDm$. Using the characterization of Ershov~\cite{Ershov75a}, Mal'tsev~\cite{Mal85} proved that the structure $\IDm$ is isomorphic to $\Dm$. Therefore, one can obtain the following:

\begin{lemma} \label{lem:mal}
	\textnormal{(essentially Mal'tsev~\cite{Mal85})}
	For any countable linear order $\mathcal{L}$ with least element, there is an initial segment of $\IDm$ isomorphic to $\mathcal{L}$.
\end{lemma}

In a way, Proposition~\ref{prop:immune} can be treated as a refinement of Lemma~\ref{lem:mal}.

\begin{proof}
We follow the Odifreddi's exposition (see \S\,6.2 in \cite{Odi92}) of the results of Lachlan~\cite{Lach70}. Suppose that $U\subseteq \omega$. For a non-empty c.e. set $W$, choose a total computable function $f$ such that $range(f) = W$. Define
\[
	\lach_U(W) = f^{-1}(U).
\]
Note that in a way, the operator $\Psi(\nu;\cdot)$ from the proof of Proposition~\ref{prop:upper-bound} is a counterpart of the operator $\Psi_U$ in the realm of numberings.

We recall some basic properties of the operator $\Psi_U$ (see pp.~561--562 in~\cite{Odi92} and p.~388 in \cite{Ershov-Book}):
\begin{itemize}
	\item[(a)] The $m$-degree of the set $\Psi_U(W)$ does not depend on the choice of $f$.
	
	\item[(b)] $\lach_U(W) \leq_m U$.
	
	\item[(c)] If $W_1 \subseteq W_2$ are c.e. sets, then $\lach_U(W_1) \leq_m \lach_U(W_2)$.
	
	\item[(d)] If $W_1$ is a c.e. set and $W_1 =^{\ast} W_2$, then $\lach_U(W_1) \equiv_m \lach_U(W_2)$.
	
	\item[(e)] For any set $X \leq_m U$, there is a c.e. set $W_0$ such that $\lach_U (W_0) \equiv_m X$.
\end{itemize}
The properties above imply that $\lach_U$ induces an epimorphism of upper semilattices, acting from $\mathcal{E}^{\ast}$ onto the ideal $\Dm(\leq U)$. We slightly abuse the notation and identify $\lach_U$ with the induced epimorphism.

W.l.o.g., we may assume that $\mathcal{L}$ is an infinite linear order, and the domain of $\mathcal{L}$ is $\omega$. Moreover, we assume that $0$ is the least element of $\mathcal{L}$, and $1$ is the greatest element of $\mathcal{L}$. Fix a strongly $D(\mathcal{L})$-computable sequence of finite linear orders $\{ L_s\}_{s\in\omega}$ such that $L_0 = \{0 < 1\}$, $\bigcup_{s\in\omega} L_s = \mathcal{L}$, $L_s \subseteq L_{s+1}$, $L_{3t}=L_{3t+1}=L_{3t+2}$, and $L_{3t}$ contains exactly $(t+2)$ elements, for any $s$ and $t$. For an element $a\in \mathcal{L}$, let $l(a) = \min \{ s\,\colon a\in L_s\}$.

At a stage $s$, we construct infinite disjoint c.e. sets $P_{a,s}$, $a\in L_s$, and a computable equivalence relation $E_s$. We also build finite sets $A_s$ and $B_s$. For $a\in L_s$, let
\[
		R_{a,s} := \bigcup_{b\leq_{L_s} a} P_{b,s}.
\]
The sets have the following properties:
\begin{itemize}
	\item $A_s \subseteq A_{s+1}$, $B_s \subseteq B_{s+1}$, $E_s \subseteq E_{s+1}$, and $A_s \cap B_s = \emptyset$,
	
	\item for any $a\in L_s$, $R_{a,s} \subseteq R_{a,s+1}$,
	
	\item $R_{1,s} \cup A_s \cup B_s = \omega$ and $R_{1,s} \cap (A_s \cup B_s) = \emptyset$,
	
	\item for any $x\in\omega$, the equivalence class $[x]_{E_s}$ is finite; moreover, for any $s$, the canonical index of $[x]_{E_s}$ can be computed uniformly in $x$;
	
	\item every $x\in\omega$ satisfies exactly one of the following:
		\begin{itemize}
			\item there is $a\in L_s$ such that $[x]_{E_s} \subseteq P_{a,s}$,
			
			\item $[x]_{E_s} \subseteq A_s$, or
			
			\item $[x]_{E_s} \subseteq B_s$;
		\end{itemize}
		
	\item if $a\in L_s$ and $x\in P_{a,s+1}$, then there is an element $y\in P_{a,s} \cap [x]_{E_{s+1}}$.
\end{itemize}
Note that the last property implies the following: for any $a\in \mathcal{L}$ and $s\geq l(a)$, the set $P_{a,s} \cap P_{a,l(a)}$ is infinite. At stage $s+1$, if we do not explicitly specify $A_{s+1}$, then we assume that $A_{s+1} = A_s$. The same applies to other sets.


\

\textbf{Intuition.} Before giving a formal construction, we discuss
the intuition behind this construction.
Recall that we want to construct an immune set $A$ such that
$\Dm(\leq A)$ is isomorphic to $\mathcal{L}$. With help of the
functional $\Psi_A$ we transfer our problem to the semilattice
$\mathcal{E}^{\ast}$. Thus, we will build c.e. sets $R_{a,
l(a)}$ which will correspond to the elements $a \in \mathcal{L}$.
In order to obtain the desired isomorphism, we must ensure
that other c.e. sets
(i.e. sets $W$ not equal to our $R_{a,l(a)}$, $a\in\mathcal{L}$)
do not induce additional $m$-degrees via the functional $\Psi_A$.
The stages $3e+1$ and $3e+2$ are devoted to this, and
Lemma~\ref{lem:odi}
ensures the correctness.

Each of c.e. sets $R_{a, s}$ consists of smaller c.e. blocks
$P_{b,s}$, where $b \leq_{L_s} a$. While satisfying requirements, we
can move some elements of a block $P_{a,s}$ only to the left blocks
(i.e. the blocks $P_{b,s}$ with $b<_{\mathcal{L}} a$)
meanwhile keeping $P_{a,s}$ infinite. Note that if we keep each
$P_{a,s}$ infinite, then it helps to distinguish $R_{a,s}$ from
$R_{b,s}$, where $b <_{L_s} a$. The equivalence classes
$[x]_{E_s}$ constitute a tool for working with $m$-reductions between
$\Psi_A(W_e)$ and $\Psi_A(R_{a, s})$. Thus, moving an equivalence
class to the left block can be considered as a correction of the
reduction.

The role of stages $3e+3$ is ensuring immunity of $A$; also at these
stages we split a block $P_{b,s}$ into two blocks $P_{a,s}$ and
$P_{b,s}$, where $a$ is a fresh element in $L_s$.

We note that the construction is non-uniform in the following
sense.
At a stage $s$,
we don't have an effective procedure to determine which elements $a$ belong to
$L_s$.
Thus, one can consider our actions as working with all possible
approximations (or variants) of $L_s$, moreover, each of these approximations works with its
own $P_{a,s+1}$. Then for a given approximation, we can uniformly
construct $P_{a,s+1}$, and this allows us to make $P_{a,s+1}$ as c.e.
and $E_{s+1}$ as computable. One can imagine this process as
working on a tree, where at a given level we have the results of the
work at the corresponding stage with different approximations. Using the oracle
$D(\mathcal{L})\oplus \emptyset^{(2)}$, we can recognize the true
path, where the real construction happens, and thus, the sets $A$ and $B$
will be c.e. relative to the approximations along this path. Now
we proceed to the formal construction.

\


\textbf{Construction.}
At stage $0$, we set $A_0 = B_0 = \emptyset$, $P_{0,0} = 2 \mathbb{N}$, $P_{1,0} = 2 \mathbb{N} + 1$, and $E_0 = id_{\omega}$.

\underline{\emph{Stage $3e+1$.}} If the set $W_e$ is finite, then we proceed to the next stage. Assume that $W_e$ is infinite. Let $s = 3e$. We find the $\leq_{\mathcal{L}}$-greatest element $a\in L_{s}$ such that $W_e \cap P_{a,s}$ is infinite. For $b>_{L_s} a$, set $P_{b,s+1} = P_{b,s}$.

The sets $P_{b,s+1}$, $b\leq_{L_s} a$, and the relation $E_{s+1}$ are constructed as follows. At step $0$, set $P^0_{b} = P_{b,s}$, $b\leq_{L_s} a$, and $\tilde E^0 = E_s$. At an odd step $(t+1)$, find (some) elements $u\in W_e\cap P^t_{a}$ and $v\in P^t_a$ such that $[u]_{\tilde E^t} = [u]_{E_s}$, $[v]_{\tilde E^t} = [v]_{E_s}$, and $[u]_{E_s} \neq [v]_{E_s}$. Set $\tilde E^{t+1} = \tilde E^t \cup ([u]_{E_s} \cup [v]_{E_s})^2$ and $P^{t+1}_b = P^t_b$ for $b\leq_{L_s} a$.

At an even step $(t+1)$, find (some) elements $u\in W_e\cap P^t_{a}$ and $v\in \bigcup_{b\leq_{L_s} a} P^t_b$ such that $[u]_{\tilde E^t} = [u]_{E_s}$, $[v]_{\tilde E^t} = [v]_{E_s}$, $[u]_{E_s} \neq [v]_{E_s}$, and $[v]_{E_s}\cap W_{e,t} = \emptyset$. Suppose that $v\in P^t_c$. For $b\leq_{L_s} a$, define
\begin{gather*}
	\tilde E^{t+1} = \tilde E^t \cup ([u]_{E_s} \cup [v]_{E_s})^2,\\
	P^{t+1}_b =
	\begin{cases}
			P^t_a \setminus [u]_{E_s}, & \text{if~} b=a\neq c\\
			P^t_c \cup [u]_{E_s}, & \text{if~} b=c\neq a,\\
			P^t_b, & \text{otherwise}.
	\end{cases}
\end{gather*}
We set $E_{s+1} = \bigcup_{t\in\omega} \tilde E^t$ and $P_{b,s+1} = \lim_{t} P^t_b$.

It is not hard to show that the sets $P_{b,s+1}$, $b\neq a$, are c.e. Note that $P_{a,s+1} \subseteq P_{a,s}$. Furthermore, we may assume that for any $x\in P_{a,s}$, there is a step $t$ such that $[x]_{\tilde E^t} = [x]_{E_s}$, $[x]_{\tilde E^{t+1}} = [x]_{E_{s+1}} \supsetneq [x]_{E_s}$, and one of the following holds:
\begin{itemize}
	\item $[x]_{E_{s+1}} \subseteq P^{t+1}_a \cap P_{a,s+1}$, or
	
	\item $[x]_{E_{s+1}} \subseteq P^{t+1}_b \cap P_{b,s+1}$ for some $b <_{L_s} a$.
\end{itemize}
Therefore, the set $P_{a,s+1}$ is c.e. and $E_{s+1}$ is computable.

The stage $3e+1$ ensures the following: If $W_e$ is infinite, then for any $u\in R_{a,s+1}$, the intersection $[u]_{E_{s+1}} \cap W_e$ is not empty.

\underline{\emph{Stage $3e+2$.}} Assume that $e = \langle a,b,k\rangle$ and $s = 3e+1$. Suppose that $a,b,k$ satisfy the following conditions:
\begin{itemize}
	\item[(i)] $a,b\in L_s$ and $a <_{\mathcal{L}} b$,
	
	\item[(ii)] $dom(\varphi_k) \supseteq R_{b, l(b)}$
	
	\item[(iii)] $range(\varphi_k) \subseteq R_{a, l(a)}$.
\end{itemize}
Choose a fresh witness $w\in P_{b,l(b)} \cap P_{b,s}$ (here the freshness means that no $v\in [w]_{E_s}$ has been a witness at earlier stages). If $\varphi_k(w) \in A_s$, then enumerate the class $[w]_{E_s}$ into $B$. If $\varphi_k(w) \in B_s$, then enumerate $[w]_{E_s}$ into $A$. If $\varphi_k(w) \not\in A_s \cup B_s$, then enumerate $[w]_{E_s}$ into $A$ and $[\varphi_k(w)]_{E_s}$ into $B$, respectively. These actions ensure the following condition: $(w\in A) \Leftrightarrow (\varphi_k(w) \not\in A)$. For $a\in L_s$, set $P_{a,s+1} = P_{a,s} \setminus ([w]_{E_s} \cup [\varphi_k(w)]_{E_s})$.

If $a,b,k$ do not satisfy the conditions (i)--(iii) above, then proceed to the next stage.

\underline{\emph{Stage $3e+3$.}} First, let $s = 3e+2$ and let $a$ be the (unique) element from $L_{s+1}\setminus L_s$. Suppose that $b$ is the least element from $L_{s}$ such that $a<_{\mathcal{L}} b$. Choose a computable enumeration $\{x_m\}_{m\in\omega}$ such that $\bigcup_{m\in\omega} [x_m]_{E_s} = P_{b,s}$ and $[x_m]_{E_s} \neq [x_k]_{E_s}$ for $m \neq k$. For $c\in L_{s+1}$, define
\[
	P^0_{c,s+1} =
		\begin{cases}
			\bigcup_{k\in\omega} [x_{2k}]_{E_s}, & \text{if~} c=a,\\
			\bigcup_{k\in\omega} [x_{2k+1}]_{E_s}, & \text{if~} c=b,\\
			P_{c,s}, & \text{otherwise}.
		\end{cases}
\]

If the set $W_e$ is finite, then set $P_{c,s+1} = P^0_{c,s+1}$ and proceed to the next stage. If $W_e$ is infinite, then choose a fresh witness $w\in W_e \setminus A_s$. Enumerate $[w]_{E_s}$ into $B$ and define $P_{c,s+1} = P^0_{c,s+1} \setminus [w]_{E_s}$ for $c\in L_{s+1}$.
\medskip

\textbf{Verification.} Let $A = \bigcup_{s\in \omega} A_s$ and $B = \bigcup_{s\in\omega} B_s$. We prove a series of lemmas.

\begin{lemma}
	Suppose that $s\in\omega$.
	\begin{enumerate}
		\item For any $a\in L_s$, the set $P_{a,s}$ is infinite c.e.
		
		\item Given $x\in\omega$, one can effectively find the canonical index of the finite set $[x]_{E_s}$.
	\end{enumerate}
	In particular, these properties imply that the sets $R_{a,s}$, $a\in L_s$, are c.e., and $E_s$ is computable.
\end{lemma}
\begin{proof}
	The proof is by induction on $s$. Assume that $L_s = \{a_0 < a_1 < \ldots <a_n\}$, $P_{a_i,s} = W_{e_i}$ for $i\leq n$, and $\varphi_j$ is a total computable function such that for any $x$, $\varphi_j(x)$ is the canonical index of $[x]_{E_s}$. Consider the case $s=3e$. Suppose that $W_e$ is an infinite set. The stage $(3e+1)$ describes an effective procedure such that given indices $e,e_0,e_1,\ldots,e_n,j$ and a number $a\in L_s$ with infinite $W_e\cap P_{a,s}$, it produces c.e. indices for the sets $P_{a_i,s+1}$ and a computable index for the function which calculates canonical indices of~$[x]_{E_{s+1}}$. The stages $(3e+2)$ and $(3e+3)$ describe similar effective procedures.
\end{proof}

\begin{lemma}
	The sets $A$ and $B$ are c.e. in $(D(\mathcal{L})\oplus \emptyset^{(2)})$.
\end{lemma}
\begin{proof}
	First, note that each of the following conditions on $e,k\in\omega$ is $\Pi^0_2$:
	\begin{itemize}
		 \item $W_e$ is infinite,
		
		 \item $W_e \subseteq dom(\varphi_k)$,
		
		 \item $range(\varphi_k)\subseteq W_e$.
	\end{itemize}
	Thus, in order to construct the enumerations $\{A_{s}\}_{s\in\omega}$ and $\{ B_s\}_{s\in\omega}$, it is sufficient to use the oracle $(D(\mathcal{L})\oplus \emptyset^{(2)})$.
\end{proof}

\begin{lemma}
	The set $A$ is immune.
\end{lemma}
\begin{proof}
	For an infinite c.e. set $W_e$, the step $(3e+3)$ ensures that $W_e \cap B \neq \emptyset$. Since $A\cap B = \emptyset$, we have $W_e\not\subseteq A$.
\end{proof}

\begin{lemma}\label{lem:odi}
	\begin{enumerate}
		\item For any $a\in\mathcal{L}$ and $s\geq l(a)$, we have $\lach_A(R_{a,s}) \equiv_m \lach_A(R_{a,l(a)})$.
		
		\item Assume that $W_e$ is an infinite c.e. set, and $a$ is the $\leq_{\mathcal{L}}$-greatest element from $L_{3e}$ such that the set $W_e \cap P_{a,3e}$ is infinite. Then $\lach_A(W_e) \equiv_m \lach_A(R_{a,3e+1})$.
		
		\item If $a,b\in\mathcal{L}$ and $a\neq b$, then $\lach_A(R_{a,l(a)}) \not\equiv_m \lach_A (R_{b,l(b)})$.
	\end{enumerate}
\end{lemma}
\begin{proof}
	The proof is essentially the same as in \cite[Proposition VI.2.2]{Odi92}. Here we only note that the stage $3e+1$ ensures the second condition, and stages $3\langle a,b,e\rangle+2$ ensure the third condition. In addition, we have $R_{a,3e+3} =^{\ast} R_{a,3e+2}$ and $\lach_A(R_{a,3e+3}) \equiv_m \lach_A(R_{a,3e+2})$.
\end{proof}

Lemma~\ref{lem:odi} shows that the function
\[
	F\colon a\in\mathcal{L} \mapsto \Psi_A(R_{a,l(a)})
\]
is an isomorphism from $\mathcal{L}$ onto $\Dm(\leq A)$.
This concludes the proof of Proposition~\ref{prop:immune}.
\end{proof}


\end{document}